\documentclass[10pt,leqno]{amsart}
\usepackage{graphicx}
\usepackage{indentfirst,csquotes}

\usepackage{amssymb,amsthm,amsmath,mathtools}
\usepackage{xcolor,paralist,hyperref,fancyhdr,etoolbox}
\allowdisplaybreaks 
\usepackage[english]{babel}
\newtheorem{theorem}{Theorem}[section]

\newtheorem{lemma}[theorem]{Lemma}
\newtheorem{proposition}[theorem]{Proposition}
\newtheorem{corollary}[theorem]{Corollary}

\theoremstyle{remark}

\numberwithin{equation}{section}

\hypersetup{ colorlinks=true, linkcolor=blue, filecolor=blue, urlcolor=blue }

\usepackage[doi,citestyle=ieee,backend=biber,sorting=nty,maxbibnames=10]{biblatex}
\bibliography{bibliographyFockSpaces.bib}
\AtEveryBibitem{
 \clearfield{date}
 \clearfield{eprint}
 \clearfield{isbn}
 \clearfield{issn}
 \clearlist{location}
 \clearfield{month}
 \clearfield{doi}
\clearfield{url}
\clearfield{zbmath}
\clearfield{zbl}
 \ifentrytype{book}{}{
  \clearlist{publisher}
  \clearname{editor}
  \clearfield{booktitle}
 }
}
\usepackage{mdframed}

\newcommand{\ph}{\mathcal{PH}_\alpha^p(\mathbb{C}^n)}
\newcommand{\hph}{\mathcal{PH}_\alpha^2(\mathbb{C}^n)}
\newcommand{\Kph}{K_{ph,\alpha}}
\newcommand{\kph}{k_{ph,\alpha}}
\newcommand{\wto}{\stackrel{w}{\to}}
\DeclareMathOperator{\tr}{tr}
\begin{document}
\title[Toeplitz operators on pluriharmonic Fock spaces]{Toeplitz operators on pluriharmonic Fock spaces} 
\author{Vladan Jaguzović}
\address{ Faculty of Natural Sciences and Mathematics, University of Banja Luka,
	Mladena Stojanovića 2,
	78000 Banja Luka,
	Republic of Srpska,
	Bosnia and Herzegovina
}
\email{vladan.jaguzovic@pmf.unibl.org}
\author{Djordjije Vujadinovi\'c}
\address{Faculty of Natural Sciences and Mathematics, University of
	Montenegro, Cetinjski put b.b. 81000 Podgorica, Montenegro}
\email{djordjijevuj@ucg.ac.me}

\begin{NoHyper} 
	\let\thefootnote\relax
\end{NoHyper}
\begin{abstract}
	In this paper, we study Toeplitz operators with a positive symbol on pluriharmonic Fock spaces over $\mathbb{C}^{n}.$ We characterize the conditions under which the Toeplitz operator $T_\mu$ is  bounded, compact, or belongs to the Schatten class $S_1$. Furthermore, we give a necessary condition for a Toeplitz operator to belong to $S_p$ for $p \geq 1,$ and under an additional assumption, we prove the sufficiency of the condition.
\end{abstract} 
\maketitle \pagestyle{myheadings}
\section{Introduction}
By $L_\alpha^p$ we denote the space of all Lebesgue measurable functions $f$ on $\mathbb{C}^{n}$ for which
\[
	\left\lVert f \right\rVert _{p,\alpha}^{p}=\frac{p^{n}}{\left( 2\pi \alpha \right) ^{n}}\int_{\mathbb{C}^{n}}^{}\left| f(z)e^{-\frac{|z|^2}{2\alpha}	} \right|^{p} \, dA(z) < \infty.
\]

The holomorpic Fock space $\mathcal{F}_\alpha^{p}$ is the closed subset of $L_\alpha^p$ consisting of all holomorphic functions belonging to $L_\alpha^p.$ These spaces play a central role in quantum physics, with their applications discussed in \cite{perelomovGeneralizedCoherentStates1977}. It is known that the reproducing kernel of $\mathcal{F}_\alpha^{2}$ is given by $K_\alpha(z,w)=e^{\frac{z \bar{w}}{\alpha}}$ (see \cite{KeheZhuAnalysisOnFockSpaces}). Moreover, the projection $P_\alpha : L_\alpha^p \to \mathcal{F}_\alpha^{p},$ given by
\[
	P_\alpha f(z)=\frac{1}{\pi \alpha}\int_{\mathbb{C}^n}^{}f(w) e^{\frac{z \bar{w}}{\alpha}}e^{-\frac{|z|^2}{\alpha}}\, dA(z)
\]
is bounded for $p \geq 1$ (see \cite{KeheZhuAnalysisOnFockSpaces}). Similarly, the boundedness of the orthogonal projection from $L_\alpha^{p}$ to the harmonic Fock space was established in \cite{vujadinovicBoundednessoftheorthogonal2022}.

The main focus of this paper is the study of Toeplitz operators acting on $\hph.$ We will provide necessary and sufficient conditions for such operators to be bounded, compact, and belong to the trace class. Those problems for holomorphic Fock space $\mathcal{F}_\alpha^2$ are considered in \cite{isralowitzToeplitzOperatorsFock2010}.  To prove our boundedness and compactness results, we will derive a characterization of Carleson measures analogous to the holomorphic case presented in \cite{KeheZhuAnalysisOnFockSpaces} and the harmonic case on $\mathbb{C}$ developed in \cite{DjordjijeCarlesonMeasures2021}.

It remains an open problem to find sufficient conditions for the operator $T_\mu$ to be in the Schatten class $S_p.$
We conjecture that for a positive Borel measure $\mu,$ the operator $T_\mu \in S_p$ if and only if $\mu(B(\cdot,r)) \in L^p(dA)$ for every $r>0.$ While we have yet to provide a complete proof of this statement, we establish the necessity of the condition $\mu(B(\cdot,r)) \in L^p(dA).$ Furthermore, we prove the sufficiency of this condition under the additional assumption that $\mu$ is a positive regular radial Borel measure.

\section{Preliminares}
A pluriharmonic function \(f: \mathbb{C}^{n}\to \mathbb{C}\) is a function such that
\[
	\frac{\partial ^2 f}{\partial z_j \partial \overline{z}_k}= 0,
\]
for all \(j,k = \{1,\dots,n\}.\) It is known that for each such function \(f,\) there exist unique holomorphic functions $g,h$  on $\mathbb{C}^{n}$  such that $h(0)=0$  and $f=g+\overline{h}.$ We denote by $\ph$ the set of all pluriharmonic functions belonging to $L^p_\alpha.$ It is easy to see that $\ph$ is a closed subspace of $L^p_\alpha.$

\subsection{Reproducing kernel}

In the paper \cite{FulscheRobertToeplitzOperatorsonPluriharmonicFunctionSpaces} it is stated that the reproducing kernel of $\hph$ is given by
\[
	\Kph(z,w)=K_{\alpha}(z,w)+K_{\alpha}(w,z)-1,
\]
where $K_{\alpha}(z,w)=e^{ \frac{z \bar{w}}{\alpha}}$ is the kernel for holomorphic Fock space $\mathcal{F}^2_\alpha.$
We also define the normalized pluriharmonic reproducing kernel $\kph(z,w)$ by
\[
	k_{ph,\alpha}(z,w)=\frac{\Kph(z,w)}{\left\lVert \Kph(\cdot,w) \right\rVert _{2,\alpha}} = \frac{\Kph(z,w)}{\sqrt{\Kph(w,w)}}= \frac{\Kph (z,w)}{\sqrt{2e^{\frac{|w|^2}{\alpha}}-1}}.
\]
Furthermore, it is easy to show that $\Kph(z,w)=\sum_{n=1}^{\infty} e_n(z)\overline{e_n(w)},$ for any  orthonormal basis $\{e_n\}$ in $\hph,$ following the same argument as in \cite[Theorem 4.19]{zhuOperatorTheoryFunction2007}.

The set $\mathrm{Lin}\, \left\{ \Kph(z,w) \mid w \in \mathbb{C}^{n} \right\} $ is a dense subset of $\hph.$ Indeed, if $f$ is orthogonal to the set, we have
\[
	f(z)=\frac{1}{\left( \alpha \pi \right) ^{n}}\int_{\mathbb{C}^{n}}^{}f(w) \Kph(z,w) e^{-\frac{|w|^2}{\alpha}}\, dA(w)= \langle f,\Kph(\cdot,z) \rangle = 0.
\]
Moreover, the set $\mathrm{Lin} \ \{\kph (z,w) \mid w \in \mathbb{C}^{n}\}$ is dense in $\hph.$ Therefore, $f_m \wto 0,$ if and only if $(f_m)$ is bounded and  $\langle f_m,\kph(\cdot,w) \rangle \to 0,$ for every $w \in \mathbb{C}^{n},$ by the Banach-Steinhaus theorem.

\subsection{Some inequalities for pluriharmonic functions}
In this subsection we will give some useful estimates for the integral average of holomorphic and harmonic functions. The first lemma, we state is a standard lemma found in almost all textbooks and papers on Fock spaces. In the second lemma, we deal with an appropriate estimate for harmonic functions.
\begin{lemma}\label{lm:point-evaluation-on-arbitrary-balls}
	Let $0<p<\infty,\alpha>0$ and $R>0$ are positive fixed paramethers. There exists a positive constant $C(p,\alpha,R)$ such that
	\[
		\left| g(a)e^{-\frac{|a|^2}{2\alpha}} \right| ^{p}\leq \frac{C}{r^{2n}}\int_{B(a,r)}^{}\left| g(z) \right| ^{p}e^{-\frac{|z|^2p}{2\alpha}}\, dA(z)
	\]
	for every $0<r<R, a \in \mathbb{C}^{n}$ and every entire function $g$ defined on $\mathbb{C}^{n}.$
\end{lemma}
\begin{proof}
	It follows by applying the same idea from \cite[Lemma 2.32]{KeheZhuAnalysisOnFockSpaces} for $C=2n e^{\frac{Rp}{2\alpha}}.$
\end{proof}
\begin{lemma}\label{lm:local-point-evaluation-estimate-for-harmonic-functions}
	Let $f$ be a harmonic function and $\alpha>0,p>0,$ and $R$ are positive fixed parametheres. Then, there exists a positive constant $C(|a|,p,\alpha,R)$ such that
	\[
		\left| f(a) e^{-\frac{|a|^2}{2\alpha}} \right| ^{p}\leq \frac{C}{r^{2n}}\int_{B(a,r)}^{}\left| f(z)e^{-\frac{|z|^2}{2\alpha}} \right| ^{p}\, dA(z)
	\]
	holds for all $0<r\leq R.$
	The constant $C$ is given by
	\[
		C=C_p e^{p\frac{R^2+2R|a|}{2\alpha}}
	\]
	where $C_p$ depends only on $p.$
\end{lemma}
\begin{proof}
	Using the inequality
	\begin{align*}
		\int_{B(a,r)}^{}\left| f(z)e^{-\frac{|z|^2}{2\alpha}} \right| ^{p}\, dA(z) & \geq e^{-p\frac{(|a|+r)^2}{2\alpha}} \int_{B(a,r)}^{}|f(z)|^p\, dA(z)
	\end{align*}
	and by Lemma \cite[Lemma 2, p. 172]{feffermanHpSpacesSeveral1972}, we obtain
	\[
		\int_{B(a,r)}^{}\left| f(z)e^{-\frac{|z|^2}{2\alpha}} \right| ^{p}\, dA(z) \geq e^{-p\frac{(|a|+r)^2}{2\alpha}} |f(a)|^p \frac{r^{2n}}{C_p} = \frac{r^{2n}}{C_p} e^{-p\frac{R^2+2R|a|}{2\alpha}} \left| f(a) e^{-\frac{|a|^2}{2\alpha}} \right| ^{p}. \qedhere
	\]
\end{proof}
\subsection{The Berezin transform of operators}
Here we introduce the Berezin transform with respect to the pluriharmonic kernel. Moreover, we establish several basic properties of this transform, for latter use.

For every bounded linear operator $T : \hph \to \hph,$ we define the Berezin transform (the Berezin symbol) of $T$ as
\[
	\widetilde{T}(z)=\langle T \kph (\cdot,z),\kph(\cdot,z) \rangle.
\]
\begin{proposition}
	If $T$ is a compact operator, then $\widetilde{T}(z) \to 0,$ as $|z| \to \infty.$
\end{proposition}
\begin{proof}
	Notice
	\[
		\langle \kph(\cdot,z),\kph(\cdot,w) \rangle = \frac{1}{\left( 2e^{\frac{|z|^2}{\alpha}}-1 \right) ^{1/2}\left( 2 e^{\frac{|w|^2}{\alpha}}-1 \right) ^{1/2}} \Kph(z,w) \to 0,
	\]
	as $|z|\to 0,$ therefore $\kph(\cdot,z)$ converges weakly to $0$ as $z \to \infty.$ Since $|\widetilde{T}(z)|\leq \left\lVert T \kph(\cdot,z) \right\rVert $, by \cite[Thm 1.14]{zhuOperatorTheoryFunction2007} we have that $\widetilde{T}(z)\to 0,$ as $|z|\to \infty.$
\end{proof}
\begin{theorem}\label{tm:berezin-transform-trace-formula}
	If $T$ is a positive operator or $T$ is a trace-class operator then
	\[
		\tr (T)= \frac{1}{\left( \alpha \pi \right) ^{n}}\int_{\mathbb{C}^{n}}^{}\widetilde{T} (z) \left( 2e^{\frac{|z|^2}{\alpha}}-1 \right) e^{-\frac{|z|^2}{\alpha}}\, dA(z).
	\]
	Also, a positive operator $T$ belongs to the trace-class if and only if the integral converges.
\end{theorem}
\begin{proof}
	Analogous to the proof in \cite[Prop. 3.3]{KeheZhuAnalysisOnFockSpaces}.
\end{proof}
\begin{proposition} \label{prop:schatten-class-berezin-transform}
	Let $p \geq 1$ and $T \in S_p.$ Then $\widetilde{T}$ belongs to $L^p(\mathbb{C}^{n},( 2e^{\frac{|z|^2}{\alpha}}-1 ) e^{-\frac{|z|^2}{\alpha}}dA(z)).$
\end{proposition}
\begin{proof}
	Notice that if $T$ is a bounded operator, that is, $T \in S_{\infty},$ we have $|\widetilde{\mu}(z)| \leq \left\lVert T \right\rVert. $ Also, if $T \in S_1$ is positive, we have that
	\[
		\frac{1}{\left( \alpha \pi \right) ^{n}}\int_{\mathbb{C}^{n}}^{}\widetilde{T}(z) \left( 2e^{\frac{|z|}{\alpha}}-1 \right) e^{-\frac{|z|}{\alpha}}\, dA(z) = \tr(T)=\left\lVert T \right\rVert _{S_{1}}.
	\]
	If $T \in S_1 $ is a self-adjoint then $T=T_1 -T_2,$ where $T_1 =\frac{|T|+T}{2},$ and $T_2 = \frac{|T|-T}{2}$ are positive operators. Since $|\widetilde{T}(z)|\leq \widetilde{T_1 }(z)+\widetilde{T_2 }(z),$ we have
	\begin{align*}
		\frac{1}{\left( \alpha \pi \right) ^{n}} \int_{\mathbb{C}^{n}}^{}\left| \widetilde{T}(z) \right| \frac{2e^{\frac{|z|^2}{\alpha}}-1 }{ e^{\frac{|z|^2}{\alpha}}} \, dA(z) \leq \frac{1}{\left( \alpha \pi \right) ^{n}}\int_{\mathbb{C}^{n}}^{} \widetilde{|T|}(z)\frac{ 2e^{\frac{|z|^2}{\alpha}}-1 }{ e^{\frac{|z|^2}{\alpha}}} \, dA(z) = \tr(|T|)=\left\lVert T \right\rVert _{S_1}.
	\end{align*}
	If $T \in S_1$ is an arbitary operator, then $T=T_1 + i T_2 = \frac{T+T^*}{2}+i \frac{T-T^*}{2i},$ where $T_1$ and $T_2 $ are self-adjoint operators. Since
	\(\left| \widetilde{T}(z) \right| \leq \left| \widetilde{T_1 } \right| + \left| \widetilde{T_2 } \right|,\) by the previous discussion, we have
	\[
		\frac{1}{\left( \alpha \pi \right) ^{n}} \int_{\mathbb{C}^{n}}^{}\left| \widetilde{T}(z) \right| \frac{2e^{\frac{|z|^2}{\alpha}}-1 }{ e^{\frac{|z|^2}{\alpha}}} \, dA(z) \leq\left\lVert T_1  \right\rVert +\left\lVert  T_2  \right\rVert \leq  2  \left\lVert  T \right\rVert.
	\]
	From the complex interpolation we have if $T \in S_p$ then $T \in L^p(\mathbb{C}^{n},\frac{2e^{\frac{|z|^2}{\alpha}}-1}{e^{\frac{|z|^2}{\alpha}}}dA(z)).$
\end{proof}
\begin{corollary}
	If $p\geq 1,$ and $T \in S_p.$ Then $\widetilde{T} \in L^p(\mathbb{C}^{n},dA(z)).$
\end{corollary}
For a positive operator on a Hilbert space, by \cite[Prop. 1.31]{zhuOperatorTheoryFunction2007}, we have $\langle T^{p}x,x \rangle\leq \langle Tx,x \rangle^{p},$ for every unit vector $x.$ This implies that for a positive operator we have $ \widetilde{T^p}\leq \widetilde{T}^p.$
\subsection{\texorpdfstring{$p$}{p}-Carleson measures}
For a non-negative Borel measure $\mu$ on $\mathbb{C}^{n}$  and $1\leq p<\infty,$ we say that $\mu$  is $p$-Carleson measure for the pluriharmonic Fock space $\mathcal{PH}_\alpha^{p}(\mathbb{C}^{n})$ if there exists a constant $C>0$  such that
\[
	\left( \int_{\mathbb{C}^{n}}^{}\left| f(z) e^{-\frac{|z|^2}{2\alpha}} \right| ^{p}\, d\mu(z)  \right) ^{1/p}\leq C \left\lVert f \right\rVert _{p,\alpha}.
\]
Here we give a characterization of the $p$-Carleson measures. These results will yield a characterisation of boundedness of Toeplitz operators on $\hph.$ For the mentioned characterisation, we will use only the case $p=2.$ However, since the same approach applies to all $p\geq 1,$ we prove the result for the general case.  Before stating the theorem, we briefly review the tools required for the proof.

Let us denote with $P_{p,\alpha} :  L_\alpha^p \to \mathcal{F}^{p}_\alpha$ projection onto the holomorphic Fock space $\mathcal{F}^p_\alpha.$ One can easy show that $P_{p,\alpha}$ is bounded, using the same idea as in  \cite[Thm 2.20]{KeheZhuAnalysisOnFockSpaces}. Furthermore, the same proof give us that $\left\lVert P_{\alpha,p}\right\rVert \leq 2^n.$

\begin{theorem}\label{tm:characterization-of-Carleson-measures}
	Let $\mu$ be a positive Borel measure on $\mathbb{C}^{n},$ and $p\geq 1,\alpha >0$ and $0<r<\infty$  fixed. Then the following conditions are equivalent:
	\begin{enumerate}
		\item $\mu$  is a $p$-Carleson measure;
		\item There exists a constant $C>0$  such that $\mu(B(a,r))\leq C$ for every $a \in \mathbb{C}^{n}.$
	\end{enumerate}
\end{theorem}
\begin{proof}
	Let $\mu$ be a $p$-Carleson measure. Then for  $f(z)=e^{\frac{z \bar{a}}{\alpha}-\frac{|a|^2}{2\alpha}},$ we obtain
	\[
		\int_{B(a,r)}^{}\left| f(z)e^{-\frac{|z|^2}{2\alpha}} \right| ^{p}\, d\mu(z)\leq \int_{\mathbb{C}^{n}}^{}e^{-\frac{p|z-a|^2}{2\alpha}}\, dA(z)\leq  C.
	\]
	Since $\left| f(z)e^{-\frac{|z|^2}{2\alpha}} \right| ^p= e^{-\frac{p|z-a|^2}{2\alpha}}\geq e^{-\frac{pr^2}{2\alpha}},$ we conclude that
	\[
		\mu(B(a,r)) \leq C e^{\frac{pr^2}{2\alpha}}.
	\]

	Let $\mu$ satisfies the second condition. For $r>0$  and lattice $r\mathbb{Z}^{2n}\subset \mathbb{C}^{n}$ one can see that there exists $M$  such that for every $z \in \mathbb{C}^{n}$ belongs to at the most $M$ balls $B(z_k,2r),$ where $\left\{ z_k \right\} _k = r\mathbb{Z}^{2n}.$ Then for every holomorphic function $g \in \ph,$ we have
	\begin{align*}
		\int_{\mathbb{C}^{n}}^{}\left| g(z) e^{-\frac{|z|^2}{2\alpha}}\right| ^{p}\, d\mu(z)
		 & \leq \sum_{z_k}^{} \int_{B(z_k,r)}^{}\left| g(z)e^{-\frac{|z|^2}{2\alpha}} \right| ^{p}\, d\mu(z)                                              \\
		 & \leq \frac{C}{r^{2n}} \sum_{z_k}^{} \int_{B(z_k,r)}^{}\int_{B(z_k,2r)}^{}\left| g(w)e^{-\frac{|z|^2}{2\alpha}} \right| ^{p}\, dA(w) \, d\mu(z) \\
		 & \leq \frac{C}{r^{2n}}\sum_{z_k}^{} \mu(B(z_k,r))\int_{B(z_k,2r)}^{}\left| g(w)e^{-\frac{|w|^2}{2\alpha}} \right| ^{p}\, dA(w)                  \\
		 & \leq \frac{C e^{\frac{pr}{2\alpha}}}{r^{2n}}\sum_{z_k}^{} \int_{B(z_k,2r)}^{}\left| g(w)e^{-\frac{|w|^2}{2\alpha}} \right| ^{p}\, dA(z)        \\
		 & \leq \frac{M C e^{\frac{pr}{2\alpha}}}{r^{2n}} \int_{\mathbb{C}^{n}}^{}\left| g(w)e^{-\frac{|w|^2}{2\alpha}} \right| ^{p}\, dA(z) .
	\end{align*}
	Let $f =g+\bar{h} \in \ph$ where $g,h \in \ph$ are entire functions such that $h(0)=0.$ Then
	\begin{align*}
		     & \int_{\mathbb{C}^{n}}^{}\left| f(z)e^{-\frac{|z|^2}{2\alpha}} \right|^{p} \, d\mu(z)                                                                                              \\ \leq                                                                                                                    & \int_{\mathbb{C}^{n}}^{}\left| g(z)e^{-\frac{|z|^2}{2\alpha}} \right| ^p\, d\mu(z) + \int_{\mathbb{C}^{n}}^{}\left| h(z)e^{-\frac{|z|^2}{2\alpha}} \right| ^{p}	\, d\mu(z) \\
		\leq & \frac{M C e^{\frac{pr}{2\alpha}}}{r^{2n}} \int_{\mathbb{C}^{n}}^{}\left| g(z)e^{-\frac{|z|^2}{2\alpha}} \right| ^{p} + \left| h(z)e^{-\frac{|z|^2}{2\alpha}} \right| ^{p}\, dA(z) \\
	\end{align*}
	By boundedness of $P_{p,\alpha}$ we have that $\left\lVert g \right\rVert _{p,\alpha} = \left\lVert P_{p,\alpha} f \right\rVert _{p,\alpha} \leq \left\lVert P_{p,\alpha} \right\rVert  \left\lVert f \right\rVert _{p,\alpha}$ and
	\[
		\left\lVert h \right\rVert _{p,\alpha} = \left\lVert \overline{P_{p,\alpha}(\bar{f})}-\overline{f(0)} \right\rVert \leq (\left\lVert P_{p,\alpha} \right\rVert +1) \left\lVert f \right\rVert _{p,\alpha}.
	\]
	As a result, $\mu$  is a $p$ -Carleson measure.
\end{proof}
For a positive Borel measure $\mu$ on $\mathbb{C}^{n}$ we say that it is a vanishing $p$-Carleson measure if
\begin{equation}\label{eq:definition-of-a-vanishing-Fock-Carleson-measures}
	\lim_{m \to \infty} \int_{\mathbb{C}^{n}}^{}\left| f_m(z)e^{-\frac{|z|^2}{2\alpha}} \right|^{p} \, d\mu(z) = 0,
\end{equation}
for every sequence $\{f_m\},$ bounded in $\ph$ such that converges to 0 uniformly on compact subsets. The following results will give us a characterization of the compactness of Toeplitz operators on $\hph$.
\begin{theorem}\label{thm:vanishing-Fock-Carleson-measures}
	Let $p\geq 1,\alpha>0,r>0$ and $\mu$ is a positive Borel measure on $\mathbb{C}^{n}.$ Then $\mu$ is a vanishing $p$-Carleson measure if and only if $\mu\left( B(z,r) \right) \to 0,$ as $z \to \infty.$
\end{theorem}
\begin{proof}
	Notice that $\mu$ is a vanishing $p$-Carleson measure, then \eqref{eq:definition-of-a-vanishing-Fock-Carleson-measures} holds for every bounded sequence from $\mathcal{F}_\alpha^{p}$ (holomorphic Fock space). Consequently, by \cite[Thm 3.30]{KeheZhuAnalysisOnFockSpaces}, we have that $\mu(B(z,r)) \to 0,$ as $z \to \infty.$

	Let $\mu(B(z,r))\to 0,$ as $z \to \infty, \{z_k\}=r\mathbb{Z}^{2n}$ and $f_m=g_m +h_m$ a bounded sequence from $\ph$ such that $f_m$ converges uniformly to 0 on compact subsets of $\mathbb{C}^{n}.$
	Since $\mu(B(z_k,r))\to \infty,$ when $k \to 0,$ we can choose $N$ such that $\mu(B(z_k,r))<\epsilon$ for $k > N.$ Taking into account that $\left\{ \mu(B(z_k,r)) \right\} $ is bounded and following the idea from Theorem \ref{tm:characterization-of-Carleson-measures} we obtain
	\begin{equation}\label{eq:vanishing-fock-carleson-measures-inequality-in-proof}
		\begin{split}
			\int_{\mathbb{C}^{n}}^{}\left| f_m(z)e^{-\frac{|z|^2}{2\alpha}} \right| ^{p}\, d\mu(z)
			\leq & \sum_{k=1}^{N}  \int_{B(z_k,r)}^{}|f_m(z)|^p e^{-p\frac{|z|^2}{2\alpha}}\, d\mu(z)                                                                                                                      \\
			     & + \frac{C \epsilon}{r^{2n}} \sum_{k=N+1}^{\infty} \int_{B(z_k,2r)}^{}\left( \left| g_m(z) e^{-\frac{|z|^2}{2\alpha}} \right| ^{p}+\left| h_m(z) e^{-\frac{|z|}{2\alpha}} \right| ^{p} \right) \, dA(z).
		\end{split}
	\end{equation}
	Using  Lemma \ref{lm:local-point-evaluation-estimate-for-harmonic-functions} in the first term of the previous expression, we derive the following upper bound
	\begin{align*}
		\sum_{k=1}^{N}  \int_{B(z_k,r)}^{}|f_m(z)|^p e^{-p\frac{|z|^2}{2\alpha}}\, d\mu(z) \leq & \sum_{k=1}^{N} \int_{B(z_k,r)}^{}C_p \frac{e^{\frac{r^2+2r|z|}{2\alpha}}}{r^{2n}} \int_{B(z,r)}^{}|f(w)|^p e^{-p \frac{|z|^2}{2\alpha}}\, dA(w) \, d\mu(z).
	\end{align*}
	Let $R$ be such that $\cup_{k=1}^N B(z_k,2r) \subset B(0,R)$ and let $M$ be such number that every $z \in \mathbb{C}^{n}$ is contained in most $M$ ball $B(z_k,2r).$  Then for every $z \in B(z_k,r), k \in \{1,\dots ,N\},$ the inequality $e^{\frac{r^2+2r|z|}{2\alpha}}\leq e^{\frac{r^2+2rR}{2\alpha}}$ holds; hence
	\begin{align*}
		\sum_{k=1}^{N}  \int_{B(z_k,r)}^{}|f_m(z)|^p e^{-p\frac{|z|^2}{2\alpha}}\, d\mu(z) \leq & C_p \frac{e^{\frac{r^2+2rR}{2\alpha}}}{r^{2n}} \sum_{k=1}^{N} \int_{B(z_k,2r)}^{} |f_m(w)|^p e^{-p \frac{|z|^2}{2\alpha}}\, dA(w) \\
		\leq                                                                                    & M C_p \frac{e^{\frac{r^2+2rR}{2\alpha}}}{r^{2n}} \int_{B(0,R)}^{} |f_m(w)|^p e^{-p \frac{|z|^2}{2\alpha}}\, dA(w) \to 0,
	\end{align*}
	as $m \to \infty.$ For the second term in \eqref{eq:vanishing-fock-carleson-measures-inequality-in-proof} we obtain
	\begin{align*}
		     & \frac{C \epsilon}{r^{2n}} \sum_{k=N+1}^{\infty} \int_{B(z_k,2r)}^{}\left( \left| g_m(z) e^{-\frac{|z|^2}{2\alpha}} \right| ^{p}+\left| h_m(z) e^{-\frac{|z|}{2\alpha}} \right| ^{p} \right) \, dA(z) \\
		\leq & \frac{MC \epsilon}{r^{2n}} \int_{\mathbb{C}^{n}}^{}\left( \left| g_m(z) e^{-\frac{|z|^2}{2\alpha}} \right| ^{p}+\left| h_m(z) e^{-\frac{|z|}{2\alpha}} \right| ^{p} \right) \, dA(z)                 \\
		\leq & \frac{(2 \left\lVert P_\alpha \right\rVert +1) MC \epsilon}{r^{2n}} \left\lVert f_m \right\rVert_{p,\alpha} \leq  C' \epsilon.
	\end{align*}
	Therefore, $\limsup_{m \to \infty} \int_{\mathbb{C}^{n}}^{} \left| f_m(z)e^{-\frac{|z|^2}{2\alpha}} \right| ^{p}d\mu(z) \leq C' \epsilon,$ which implies that $\mu$ is vanishing. 
\end{proof}
\begin{corollary}
	Suppose that $\mu$ is a positive Borel measure on $\mathbb{C}^{n},r>0$ and $\{z_k\}$ is any arrangement into a sequence of the lattice $r\mathbb{Z}^{2n}.$ Then
	\begin{enumerate}
		\item $\mu$ is a $p$-Carleson measure if and only if $\{\mu(B(z_k,r))\}$ is in $l^\infty.$
		\item $\mu$ is a vanishing $p$-Carleson measure if and only if the sequence $\{\mu(B(z_k,r))\}$ is in $c_0 .$
	\end{enumerate}
\end{corollary}
\section{Toeplitz operators}
For every $g \in L^{\infty}(\mathbb{C}^{n}),$ we define  $T_g : \hph \to \hph$ with
\[
	T_g(f)=P_{p,\alpha}^{ph}(fg),
\]
where $P_{p,\alpha}^{ph}$ is an orthogonal projection from $L^2(\mathbb{C}^{n},e^{-\frac{|z|^2}{\alpha}}dA(z))$ to $\hph.$ The previously defined operator we call Toeplitz operator with symbol $g.$ One can easily see that $T_g$ is bounded and $\left\lVert T_g \right\rVert \leq \left\lVert g \right\rVert _{\infty}.$ Moreover,
\[
	T_g(f) (z)= \frac{1}{\left( \alpha \pi \right) ^{n}}\int_{\mathbb{C}^{n}}^{}f(w)g(w)\Kph(z,w)e^{-\frac{|w|^2}{\alpha}}\, dA(w).
\]
The previous formula motivates the definition of Toeplitz operators for a wider class of symbols.
If $\mu$ is a Borel measure on $\mathbb{C}^{n},$ then we define
\begin{equation}\label{eq:definition-of-Toeplitz-operator-for-measures}
	T_\mu(f)(z)=\int_{\mathbb{C}^{n}}^{}\Kph(z,w) f(w)e^{- \frac{|w|^2}{\alpha}}	\, d\mu(w), \quad z \in  \mathbb{C}^{n}.
\end{equation}
In order that the previous expression is defined for $f \in \hph,$ we consider only measures such that
\begin{equation}\label{eq:condition-for-measures-for-toeplitz-operator}
	\int_{\mathbb{C}^{n}}^{}\left| \Kph(z,w) \right| ^2 e^{-\frac{|w|^2}{\alpha}}	\, d|\mu|(w)<\infty, \quad \forall z \in  \mathbb{C}^{n}.
\end{equation}
Notice that if \eqref{eq:condition-for-measures-for-toeplitz-operator} holds, we have
\begin{gather*}
	\int_{\mathbb{C}^{n}}^{}\left| \Kph (z,w) \right| e^{-\frac{|w|^2}{\alpha}}\, d|\mu|(w) \\ \leq
	\left( \int_{\mathbb{C}^{n}}^{}\left| \Kph(z,w) \right| ^{2}e^{-\frac{|w|^2}{\alpha}}\, d|\mu|(w) \right) ^{1/2} \left( \int_{}^{}e^{-\frac{|w|^2}{\alpha}}\, d|\mu|(w) \right) ^{1/2},
\end{gather*}
where the finiteness of the second factor we obtain from \eqref{eq:condition-for-measures-for-toeplitz-operator} for $z=0.$
Therefore, $T_\mu(f)$ is defined whenever $f$  is a finite linear combination of kernel functions, i.e., on the fundamental subset of $\hph.$
If $d\mu = \varphi dA$, then we simply write $T_\varphi$ instead of $T_\mu$. From now on, we will restrict our attention to Toeplitz operators with positive symbols.
\begin{proposition}\label{prop:correspondence-between-FC-measures-and-boundedness}
	Let $\mu$ be a positive Borel measure on $\mathbb{C}^{n}.$ The operator $T_\mu$ is bounded if and only if $\mu $ is a $2$-Carleson measure.
\end{proposition}

\begin{proof}
	Let $T_\mu$ is bounded. By the boundedness of $T_\mu$, we have $|\langle T_\mu f,f \rangle|\leq \left\lVert T_\mu \right\rVert \left\lVert f \right\rVert ^2$ for every $f \in \hph$, and by Fatou's lemma, we obtain
	\[
		\int_{\mathbb{C}^{n}}^{}|f(w)|^2 e^{-\frac{|w|^2}{\alpha}}\, d\mu(w)\leq \left\lVert T_\mu \right\rVert \int_{\mathbb{C}^{n}}^{}\left| f(w) \right| ^2e^{-\frac{|w|^2}{\alpha}}\, dA(w).
	\]
	Hence, $\mu$ is a $2$-Carleson measure.

	If $\mu$ is a $2$-Carleson measure, then
	\[
		\int_{\mathbb{C}^{n}}^{}|f(w)|^2 e^{-\frac{|w|^2}{\alpha}}\, d\mu(w) \leq  C \int_{\mathbb{C}^{n}}^{}\left| f(w) \right| ^2e^{-\frac{|w|^2}{\alpha}}\, dA(w),
	\]
	Using the previous inequality and Cauchy-Schwarz inequality, we obtain that
	\[
		\langle f,g \rangle_{\mu} = \int_{\mathbb{C}^{n}}^{}f(z)\overline{g(z)}e^{-\frac{|z|^2}{\alpha}}\, d\mu(z)
	\]
	is a bounded sesquilinear form on $\hph;$ therefore, there exists a bounded operator $T:\hph \to \hph$ such that $\langle Tf,g \rangle=\langle f,g \rangle_\mu.$ Using the reproducing kernels we obtain
	\[
		Tf(z)=\langle Tf,\Kph(\cdot,z) \rangle =\langle f,\Kph(\cdot,z) \rangle_\mu = \int_{\mathbb{C}^{n}}^{}f(w)\Kph(w,z)e^{-\frac{|w|^2}{\alpha}}\, d\mu(w) = T_\mu f(z),
	\]
	hence $T_\mu$ is bounded.
\end{proof}
In $\hph$ the condition \eqref{eq:definition-of-a-vanishing-Fock-Carleson-measures} is equivalent to weak convergence of the sequence of functions $(f_m)$ to the zero function. Indeed, $f_m \stackrel{w}{\to} 0$ is equivalent to $(f_m)$ being bounded and $\langle f_m,\kph(\cdot,w) \rangle\to 0,$ for every $w \in \mathbb{C}^{n};$ that is,  $\left( f_m \right) $ is bounded and $f_m(w)\to 0$ for every $w \in \mathbb{C}^{n}.$
Since
\[
	\left| f_m(z) \right| =\left| \langle f_m,\Kph(\cdot,z) \rangle \right|  \leq \left\lVert f_m \right\rVert_{2,\alpha}\cdot \left\lVert \Kph(\cdot,z) \right\rVert_{2,\alpha} \leq \left\lVert f_m \right\rVert_{2,\alpha} \sqrt{ 2e^{\frac{|z|^2}{\alpha}}-1} ,
\]
we have that $f_m$ is uniformly bounded on compact subsets of $\mathbb{C}^{n}.$ By the compactness principle (see \cite[Thm 2.2.11]{klimekPluripotentialTheory1991}), we have that $(f_m)$ is normal. Since, $f_n(w) \to 0$ for every $w \in \mathbb{C}^{n}$ and since $(f_n)$ is normal, we conclude that $f_n$ converges uniformly on compact subsets of $\mathbb{C}^{n}$.

\begin{proposition}
	Operator $T_\mu$ is a compact operator if and only if $\mu$ is a vanishing $2$-Carleson measure.
\end{proposition}
\begin{proof}
	Let $T_\mu$ be a compact operator and $\left( f_m \right) $ be a bounded sequence of functions from $\hph$, such that it converges to 0 uniformly on compact subsets, that is, by the previous discussion, $f_m \wto 0$. Since $T_\mu$ is compact, it follows that $\left\lVert T_\mu f_m \right\rVert \to 0$ (see \cite[Thm 1.14]{zhuOperatorTheoryFunction2007}), hence
	\[
		\int_{\mathbb{C}^{n}}^{}\left| f_m(z) \right| ^2e^{-\frac{|z|^2}{\alpha}}\, d\mu(z) = \langle T_\mu f_m,f_m \rangle  \to 0,
	\]
	as $m \to \infty.$ Therefore, $\mu$ is a vanishing $2$-Carleson measure.

	If $\mu$ is a vanishing $2$-Carleson measure, then $\mu$ is a $2$-Carleson measure, and $T_\mu$ is bounded. For every $f \in \hph$ we have
	\begin{align*}
		\left\lVert T_\mu f \right\rVert_{2,\alpha}^2
		 & = \langle T_\mu f,T_\mu f \rangle = \int_{\mathbb{C}^{n}}^{}T_\mu f(z) \overline{f(z)} e^{-\frac{|z|^2}{\alpha}}\, d\mu(z)                                                                                                  \\
		 & \leq\left( \int_{\mathbb{C}^{n}}^{}\left| T_\mu f(z) \right|^2e^{-\frac{|z|^2}{\alpha}} \, d\mu(z) \right) ^{1/2} \left( \int_{\mathbb{C}^{n}}^{}\left| f(z) \right| ^2 e^{-\frac{|z|^2}{\alpha}}\, d\mu(z) \right) ^{1/2}  \\
		 & \leq C \left( \int_{\mathbb{C}^{n}}^{}\left| T_\mu f(z) \right|^2e^{-\frac{|z|^2}{\alpha}} \, dA(z) \right) ^{1/2} \left( \int_{\mathbb{C}^{n}}^{}\left| f(z) \right| ^2 e^{-\frac{|z|^2}{\alpha}}\, d\mu(z) \right) ^{1/2} \\
		 & \leq  C \left\lVert T_\mu \right\rVert \left\lVert f \right\rVert _{2,\alpha} \left( \int_{\mathbb{C}^{n}}^{}\left| f (z)\right| ^{2}e^{-\frac{|z|^2}{\alpha}}\, d\mu(z)  \right)^{1/2}.
	\end{align*}
	Hence, if $(f_m)$ is a bounded sequence functions from $\hph$ such that $f_m$ converge uniformly on compact subsets of $\mathbb{C}^{n},$ that is, by the previous discussion $f_m \wto 0,$ we have that $\left\lVert T_\mu f_m \right\rVert_{2,\alpha} \to 0.$ As a consequence we conclude that $T_\mu$ is compact (see \cite[Thm 1.14]{zhuOperatorTheoryFunction2007}).
\end{proof}
\subsection{Berezin transform of a measure}
If $\mu$ is a measure satisfying condition \eqref{eq:condition-for-measures-for-toeplitz-operator}, we define the Berezin transform of $\mu$ by
\begin{align*}
	\widetilde{\mu}(z)= \int_{\mathbb{C}^{n}}^{}\frac{\left| \Kph(z,u) \right| ^2}{e^{\frac{|u|^2}{\alpha}}\left( 2e^{\frac{|z|^2}{\alpha}}-1 \right) }\, d\mu(u),\quad z \in \mathbb{C}^{n}.
\end{align*}
Also, if $T_\mu : \hph \to \hph$ defined with \eqref{eq:definition-of-Toeplitz-operator-for-measures} is bounded, then
\[
	\int_{\mathbb{C}^{n}}^{}\Kph (z,w) f(w) e^{-\frac{|w|^2}{\alpha}}\, d\mu(w) = T_\mu f(z)= \langle T_\mu f,\Kph(\cdot,z) \rangle, \quad \forall f \in \hph;
\]
hence
\(
\widetilde{\mu}(z)=\langle T_\mu \kph(\cdot,z),\kph(\cdot,z) \rangle.
\)
Then  $|\widetilde{\mu}(z)|=|\langle T_\mu\left( \kph(\cdot,z) \right) ,\kph(\cdot,z) \rangle| \leq \left\lVert T_{\mu} \right\rVert,$ hence $\widetilde{\mu}$ is a bounded function.
\subsection{Trace class}
\begin{proposition}
	Let \(\mu \geq 0.\) Then $T_\mu \in S_1$ if and only if $\mu(\mathbb{C}^{n})<\infty.$
\end{proposition}
\begin{proof}
	Since $ \mu \geq 0,$ we have that $T_\mu$ is a positive operator. Therefore, by Theorem \ref{tm:berezin-transform-trace-formula},  $T_\mu \in S_1$ if and only if
	\[
		\tr (T_\mu) = \frac{1}{\left( \alpha \pi \right) ^{n}} \int_{\mathbb{C}^{n}}^{}\widetilde{\mu}(z)\left[ 2e^{\frac{|z|^2}{\alpha}}-1 \right] e^{-\frac{|z|^2}{\alpha}}\, dA(z)< \infty.
	\]
	By Fubini's theorem
	\begin{align*}
		  & \frac{1}{\left( \alpha \pi\right) ^{n}}\int_{\mathbb{C}^{n}}^{}\widetilde{\mu}(z)\left[ 2e^{\frac{|z|^2}{\alpha}}-1 \right] e^{-\frac{|z|^2}{\alpha}}\, dA(z)                                                                                                        \\
		= & \frac{1}{\left( \alpha \pi \right) ^{n}}\int_{\mathbb{C}^{n}} \int_{\mathbb{C}^{n}}^{}\frac{|\Kph(z,u)|^2}{e^{\frac{|u|^2}{\alpha}}\left( 2e^{\frac{|z|^2}{\alpha}}-1 \right) }\, d\mu(u) \left[ 2e^{\frac{|z|^2}{\alpha}}-1 \right] e^{-\frac{|z|^2}{\alpha}} dA(z) \\
		= & \frac{1}{\left( \alpha \pi \right) ^{n}}\int_{\mathbb{C}^{n}}^{}\int_{\mathbb{C}^{n}}^{}|\Kph(z,u)|^2 e^{-\frac{|z|^2}{\alpha}}\, dA(z)e^{-\frac{|u|^2}{\alpha}} \, d\mu(u)                                                                                          \\
		= & \int_{\mathbb{C}^{n}}^{}\Kph(u,u) e^{-\frac{|u|^2}{\alpha}}\, d\mu(u)                                                                                                                                                                                                \\
		= & \int_{\mathbb{C}^{n}}^{}\left[ 2e^{\frac{|u|^2}{\alpha}}-1 \right] e^{-\frac{|u|^2}{\alpha}}\, d\mu(u).
	\end{align*}
	Since $1 \leq \frac{2e^{\frac{|u|^2}{\alpha}}-1}{e^{\frac{|u|^2}{\alpha}}}\leq 2,$ the desired conclusion follows.
\end{proof}
\subsection{Schatten class}
In this subsection, we provide some partial results about characterisation of the condition $T_\mu \in S_p$ for $p>1.$ We begin by considering absolutely continuous measures.
For every $p \in [1,\infty)$ and for every $\varphi \in L^p(\mathbb{C}^{n},dA)$, the Toeplitz operator $T_\varphi$ is bounded. Indeed, $\varphi = \varphi_1 -\varphi_2 +i \varphi_3 -i \varphi_4,$ where $\varphi_k \in L^p(\mathbb{C}^{n},dA)$ are positive functions. Then for every $a  \in \mathbb{C}^{n}$ and $r>0$ we have
\[
	\int_{B(a,r)}^{} \varphi_k(z) \, dA(z) \leq m(B(a,r))^{1/p'}  \left( \int_{B(a,r)}^{}|\varphi_k(z)|^{p}\, dA(z) \right) ^{1/p}  \leq C  \left( \int_{\mathbb{C}^{n}}^{}|\varphi(z)|^{p}\, dA(z) \right) ^{1/p} ,
\]
where $\frac{1}{p}+\frac{1}{p'}=1$. By Theorem \ref{tm:characterization-of-Carleson-measures} we have that $d\mu_k = \varphi_k dA$ is a $p$-Carleson measure: Therefore, $T_{\varphi_k}$ is a bounded operator by Proposition \ref{prop:correspondence-between-FC-measures-and-boundedness}, which implies that $T_\varphi$ is bounded.

Also, similarly we can obtain that if $\varphi \in L^1(dA)$ has compact support, then $T_\mu$ is compact.
\begin{proposition}\label{prop:toeplitz-function-operator-shatten-class}
	If $p \geq  1$ and $\varphi \in L^p(\mathbb{C}^{n},dA(z)),$ then $T_\varphi$ is in Schatten class $S_p.$
\end{proposition}
\begin{proof}
	Decomposing $\varphi$ as a linear combination of four nonnegative functions in $L^p(\mathbb{C}^{n},dA),$ we can assume that $\varphi$ is nonnegative. Assume also that $\varphi$ has compact support in $\mathbb{C}^{n}.$ Then $T_\varphi$ is a positive compact operator. Therefore, it has a canonical decomposition of the form
	\[
		T_\varphi f = \sum_{n=1}^{\infty} \lambda_n \langle f,e_n \rangle e_n,
	\]
	where $\{e_n\}$ is an orthonormal set in $\hph.$
	Following the same idea from \cite[Prop. 7.11]{zhuOperatorTheoryFunction2007}, we obtain that
	\[
		\sum_{n=1}^{\infty} |\lambda_n|^{p} \leq  \int_{\mathbb{C}^{n}}^{}|\varphi(z)|^p \Kph (z,z) e^{-\frac{|z|^2}{\alpha}}\, \frac{dA(z)}{\left( \alpha \pi  \right) ^{n}} \leq  \frac{2}{\left( \alpha \pi  \right) ^{n}} \int_{\mathbb{C}^{n}}^{}|\varphi(z)|^{p}\, dA(z).
	\]
	Therefore, $T_\varphi$ is in $S_p.$

	The general case follows from monotone convergence and approximating $\varphi$ with $\varphi \chi_{B(r,0)}.$
\end{proof}
In the remainder of this paper, we present some results supporting the conjecture stated in the introduction.

Let $M=\mathcal{F}^2(\mathbb{C}^{n})$ and $N=\mathcal{F}^2 (\mathbb{C}^{n})\ominus \mathcal{A}$, where $\mathcal{A}$ is the subspace set of constant functions. Then, we have $\hph = M \oplus N$ and $M \bot N.$
We decompose $T_\mu$ as
\[
	\begin{split}
		T_\mu (f)(z) & =\int_{\mathbb{C}^{n}}^{}f(w)e^{\frac{z \bar{w}}{\alpha}}e^{-\frac{|w|^2}{\alpha}}\, d\mu(w) + \int_{\mathbb{C}^{n}}^{}f(w) \left( e^{\frac{\bar{z} w}{\alpha}} -1 \right) e^{-\frac{|w|^2}{\alpha}}\, d\mu(w) \\
		             & = T_{M,\mu}(f)(z)+T_{N,\mu}(f)(z).
	\end{split}
\]
One can easily notice that
\[
	\begin{split}
		\langle T_{M,\mu}f,g \rangle= \langle f,P_M g \rangle_{\mu}= \int_{\mathbb{C}^{n}}^{}f(w) \overline{P_M g(w)}e^{-\frac{|w|^2}{\alpha}}\, d\mu(w); \\
		\langle T_{N,\mu}f,g \rangle= \langle f,P_N g \rangle_\mu= \int_{\mathbb{C}^{n}}^{}f(w) \overline{P_N g(w)}e^{-\frac{|w|^2}{\alpha}}\, d\mu(w).
	\end{split}
\]
\begin{proposition}
	Let $T_\mu \in S_p$ then $\mu(B(\cdot,r)) \in L^p(dA)$ for every $r>0.$
\end{proposition}
\begin{proof}
	If $(e_n)$ is a basis for $\mathcal{F}_\alpha^{2},$ then
	\[
		\left[ \sum_{n=1}^{\infty} |\langle T_\mu e_n,e_n  \rangle|^p \right] ^{1/p} \leq \left\lVert T_\mu \right\rVert_p
	\]
	By the identities mentioned earlier, we have
	\[
		\langle T_\mu e_n,e_n \rangle= \langle T_{M,\mu}e_n,e_n \rangle + \langle T_{N,\mu}e_n,e_n \rangle = \langle T_{M,\mu}e_n,e_n \rangle.
	\]
	Therefore, we conclude that $T_{M,\mu} : \mathcal{F}_\alpha^2  \to \mathcal{F}_{\alpha}^2$ belongs to the Schatten class $S_p.$ By \cite[Thm 4.4]{isralowitzToeplitzOperatorsFock2010}, it follows that $\mu(B(\cdot,r)) \in L^p(dA),$ for every $r>0.$
\end{proof}
\textbf{Sufficiency of the condition} Assume that $\mu(B(\cdot,r)) \in L^p(dA).$
Let us decompose $T_\mu = T_{M,\mu}P_M+ T_{M,\mu}P_N + T_{N,\mu}P_M + T_{N,\mu}P_N.$ Since $T_{M,\mu}: \mathcal{F}_\alpha^{2}\to \mathcal{F}_{\alpha}^{2}$ is in class $S_p,$ we have $T_{M,\mu}P_M \in S_p$ (as they have the same canonical decomposition). Futhermore,  it is straightforward to see that $T_{N,\mu}:N \to N$ is in $S_p.$ Moreover, since
\[
	\langle T_{M,\mu} P_N x,y \rangle = \langle P_N x,P_M y \rangle_{\mu} = \langle x,T_{N,\mu} P_M y \rangle,
\]
we conclude that $(T_{M,\mu} P_N)^* = T_{N,\mu} P_M,  $ and, hence it sufficies to prove that $T_M P_N$ is in $S_p.$

Assume additionally that $\mu$ is a positive regular radial Borel measure. Then, the measure $\mu$ satisfies $\int_{\mathbb{C}^{n}}^{}e^{-\frac{|w|^2}{\alpha}}\, d\mu(w)<\infty. $ Following the method from \cite[Thm 2.4]{maximenkoToeplitzOperatorsBergman2025c}, we obtain that there exists a Borel measure $\nu$ on $[0,\infty),$ such that
\[
	\int_{\mathbb{C}^{n}}^{}f(w)e^{-\frac{|w|^2}{\alpha}}\, d\mu = \int_{0}^{\infty}\int_{\mathbb{S}^{2n-1}}^{}f(r \zeta)\, d\sigma_{{2n-1}}(\zeta) \, d\nu(r).
\]
Note that if $f$ is a harmonic function, the previous formula implies that
\[
	\int_{\mathbb{C}^{n}}^{}f(w)e^{-\frac{|w|^2}{\alpha}}\, d\mu =f(0) \cdot \int_{\mathbb{C}^{n}}^{}e^{-\frac{|w|^2}{\alpha}}\, d\mu(w).
\]
In this case, we have
\[
	T_{M,\mu} P_N f(z)=\int_{\mathbb{C}^{n}}^{}e^{\frac{z \bar{w}}{\alpha}}\overline{P_N f(w)}e^{-\frac{|w|^2}{\alpha}}\, d\mu(w)=0.
\]
Therefore, $T_{M,\mu} P_N$ is the zero operator; hence, $T_{M,\mu} P_N \in S_p$ for every $p.$

\textbf{Conflict of interest.} The authors declares that they have not conflict of interests.

\textbf{Acknowledgements.} Dj. Vujadinović gratefully acknowledges financial
support from the Ministry of Education, Science and Innovation of Montenegro through the
grants “Mathematical Analysis, Optimisation and Machine Learning” and “Complex-analytic
and geometric techniques for non-Euclidean machine learning: theory and applications.”

\nocite{*}
\printbibliography
\end{document}